\documentclass[12pt]{article}
\usepackage{a4}
\usepackage{amsthm,amsmath,amssymb,setspace}
\usepackage{citesort} %%%%%%%%%%%%%%  sorts the citation
\usepackage{enumerate}
\usepackage{amsthm}
\usepackage{amsfonts}
\usepackage{amsmath}
\usepackage{mathrsfs}
\usepackage{amssymb}
\usepackage{enumerate}
\usepackage{graphicx}
\usepackage{pstricks}
\usepackage{caption}
\usepackage{multirow}
\usepackage{hhline}
\usepackage{booktabs}
\usepackage{colortbl}
\usepackage{rotating}
\usepackage[T1]{fontenc}  %%% makes font as type 1, helps to make fonts sharper %%%
\usepackage{times}     %%%% makes fonts in Times New Roman  %%%%
\usepackage[a4paper,left=2.5cm,right=2cm,top=2.5cm,bottom=2.5cm]{geometry}
\usepackage{enumitem}          %%%%%%%%%       for using [itemsep=0mm]
\setlist[1]{itemsep=-1mm} %to change line separtion space globally
\newtheoremstyle{bthm}{\baselineskip}{\baselineskip}{\slshape}{}{\bfseries}{}{ }{}
\newtheoremstyle{bex}{\baselineskip}{\baselineskip}{}{}{\sffamily}{:}{\newline }{}
\theoremstyle{bthm}
\newtheorem{thm}{Theorem}[section]
\newtheorem{cor}[thm]{Corollary}

\newtheorem{prop}[thm]{Proposition}

\newtheorem{obs}[thm]{Observation}

\newtheorem{con}[thm]{Conjecture}

\theoremstyle{bex}

\begin{document}
\begin{titlepage}
\title{On indicated coloring of lexicographic product of graphs}
\author{P. Francis$^1$, S. Francis Raj$^{2}$, M. Gokulnath$^{3}$}
\date{{\footnotesize$^{1}$Department of Computer Science, Indian Institute of Technology, Palakkad-678557, India.\\ $^{2,3}$Department of Mathematics, Pondicherry University, Puducherry-605014, India.}\\
{\footnotesize$^1$ pfrancis@iitpkd.ac.in, $^{2}$ francisraj\_s@yahoo.com\, $^{3}$ gokulnath.math@gmail.com\ }}
\maketitle
\renewcommand{\baselinestretch}{1.3}\normalsize

\begin{abstract}
Indicated coloring is a graph coloring game in which two players collectively color the vertices of a graph in the following way. In each round the first player (Ann) selects a vertex, and then  the second player (Ben) colors it properly, using a fixed set of colors. The goal of Ann is to achieve a proper coloring of the whole graph, while Ben is trying to prevent the realization of this project. The smallest number of colors necessary for Ann to win the game on a graph $G$ (regardless of Ben's strategy) is called the indicated chromatic number of $G$, denoted by $\chi_i(G)$. 
In this paper, we have shown that for any graphs $G$ and $H$, $G[H]$ is $k$-indicated colorable for all $k\geq\mathrm{col}(G)\mathrm{col}(H)$. Also, we have shown that for any graph $G$ and for some classes of graphs $H$ with $\chi(H)=\chi_i(H)=\ell$, $G[H]$ is $k$-indicated colorable if and only if $G[K_\ell]$ is $k$-indicated colorable. As a consequence of this result we have shown that if $G\in\mathcal{G}=\Big\{$Chordal graphs, Cographs, Complement of bipartite graphs, $\{P_5,C_4\}$-free graphs, connected $\{P_6,C_5,\overline{P_5}, K_{1,3}\}$-free graphs which contain an induced $C_6$, Complete multipartite graphs$\Big\}$ and $H\in\mathcal{F}=\Big\{$Bipartite graphs, Chordal graphs, Cographs, $\{P_5,K_3\}$-free graphs, $\{P_5,Paw\}$-free graphs, Complement of bipartite graphs, $\{P_5,K_4,Kite,Bull\}$-free graphs, connected $\{P_6,C_5,\overline{P_5},K_{1,3}\}$-free graphs which contain an induced $C_6$, $\mathbb{K}[C_5](m_1, m_2 ,\ldots,m_5)$, $\{P_5,C_4\}$-free graphs, connected $\{P_5,\overline{P_2\cup P_3},\overline{P_5},Dart\}$-free graphs which contain an induced $C_5\Big\}$, then $G[H]$ is $k$-indicated colorable for every $k\geq \chi(G[H])$. This serves as a partial answer to one of the questions raised by A. Grzesik in \cite{and}. In addition, if $G$ is a Bipartite graph or a $\{P_5,K_3\}$-free graph (or) a $\{P_5,Paw\}$-free graph and $H\in\mathcal{F}$, then we have shown that $\chi_i(G[H])=\chi(G[H])$. 
\end{abstract}
\noindent
\textbf{Key Words:} Game chromatic number, Indicated chromatic number, Lexicographic product of graphs. \\
\textbf{2000 AMS Subject Classification:} 05C15
%%%%%%%%%%%%%%%%%%%%%%%%%%%%%%%%%%%%%%%%%%%%%%%%%%%%%%%%%%%%%%%%%%%%%%%%%%%%%%%%%%%%%%%%%%%%%%%%%
%%%%%%%%%%%%%%%%%%%%%%%%%%%%%%%%%%%%%%%%%%%%%%%%%%%%%%%%%%%%%%%%%%%%%%%%%%%%%%%%%%%%%%%%%%%%%%%%%
%%%%%%%%%%%%%%%%%%%%%%%%%%%%%%%%%%%%%%%%%%%%%%%%%%%%%%%%%%%%%%%%%%%%%%%%%%%%%%%%%%%%%%%%%%%%%%%%%
\section{Introduction}\label{intro}

All graphs considered in this paper are simple, finite and undirected. For any positive integer $k$, a proper $k$-coloring of a graph $G$ is a mapping $c$ : $V(G)\rightarrow\{1,2,\ldots,k\}$ such that for any two adjacent vertices $u,v\in V(G)$, $c(u)\neq c(v)$. 
A graph is said to be $k$-colorable if it admits a proper $k$-coloring. 
The chromatic number $\chi(G)$ of a graph $G$ is the smallest $k$ such that $G$ is $k$-colorable. 
In this paper, $P_n, C_n$ and $ K_n$ respectively denotes the path, the cycle and the complete graph on $n$ vertices.
%For $S,T\subseteq V(G)$, let $\langle S\rangle$ denote the subgraph induced by $S$ in $G$ and let $[S,T]$ denote the set of all edges with one end in $S$ and the other end in $T$.
For any graph $G$, let $\overline{G}$ denotes the complement of $G$.

Let us recall some of the definitions which are required for this paper. 

Let $\mathcal{F}$ be a family of graphs. 
We say that a graph $G$ is $\mathcal{F}$-free if it contains no induced subgraph which is isomorphic to a graph in $\mathcal{F}$.
%For two vertex-disjoint graphs $G_1$ and $G_2$, the join of $G_1$ and $G_2$, denoted by $G_1+G_2$, is the graph whose vertex set $V(G_1+G_2) = V(G_1)\cup V(G_2)$ and
%the edge set $E(G_1+G_2) = E(G_1)\cup E(G_2)\cup\{xy: x\in V(G_1),\ y\in V(G_2)\}$.
Next, the coloring number of a graph $G$ (see \cite{jens}), denoted by $\mathrm{col}(G)$, is defined
as the smallest number $d$ such that for some linear ordering $<$ of the vertex set, the ``back degree'' $|\{y:y<x,xy\in E(G)\}|$ of every vertex $x$ is strictly less than $d$. In other words, if the vertices of $G$ are $x_1, x_2,\ldots, x_n$, then 
\begin{equation*}
\mathrm{col}(G)=1+\min\limits_p{\max\limits_i{\{d(x_{p(i)}, G_{p(i)})\}}}, 
\end{equation*}
where the minimum is taken over all permutations $p$ of $\{1,2,\ldots,n\}$ and $G_{p(i)}$ is the subgraph of $G$ induced by $x_{p(1)}, x_{p(2)},\ldots, x_{p(i)}$, and where $d(x, H)$ denotes the degree of a vertex $x$ in a graph $H$. 
It is clear that $\mathrm{col}(G)\leq\Delta(G)+1$. 
Equivalently, the coloring number can be defined as $\mathrm{col}(G)=1+\max\limits_{H\subseteq G}\delta(H)$, where $H\subseteq G$ means $H$ is a subgraph of $G$.

The lexicographic product of two graphs $G$ and $H$, denoted by $G[H]$, is a graph whose vertex set $V(G)\times V(H) =\{(x, y): x\in V(G)~\mathrm{and}~ y\in V(H)\}$ and two vertices $(x_1, y_1)$ and $(x_2,y_2)$ of $G[H]$ are adjacent if and only if either $x_1 = x_2$ and $y_1y_2 \in E(H)$, or $x_1x_2\in E(G)$.
For each $u\in V(G)$, $\langle  u\times V(H)\rangle$  is isomorphic to $H$ and it is  denoted by $H_u$  and for each $v\in V(H)$, $\langle V(G)\times v\rangle$ is isomorphic to $G$ and it is denoted by $G_v$. 

Let $G$ be a graph on $n$ vertices $v_1,v_2,\ldots,v_n$, and let $H_1,H_2,\ldots,H_n$ be $n$ vertex-disjoint graphs. 
An expansion $G(H_1,H_2,\ldots,H_n)$ of $G$ (see \cite{cho}) is the graph obtained from $G$ by\\
(i)  replacing each $v_i$ of $G$ by $H_i$, $i=1,2,\ldots,n$, and\\
(ii) by joining every vertex in $H_i$ with every vertex in $ H_j$ whenever $v_i$ and $v_j$ are adjacent in $G$.

For $i\in\{1,2,\ldots,n\}$, if $H_i\cong K_{m_i}$, then $G(H_1,H_2,\ldots,H_n)$ is said to be a complete expansion of $G$ and is denoted by $\mathbb{K}[G](m_1,m_2,\ldots,m_n)$ or $\mathbb{K}[G]$. For $i\in\{1,2,\ldots,n\}$, if  $H_i\cong \overline{K_{m_i}}$, then $G(H_1,H_2,\ldots,H_n)$ is said to be an independent expansion of $G$ and is denoted by $\mathbb{I}[G](m_1,m_2,\ldots,m_n)$ or $\mathbb{I}[G]$. It can be noted that, if $m_1=m_2=\ldots=m_n=m$, then  $\mathbb{K}[G](m_1, m_2,$ $ \ldots, m_n )\cong G[K_m]$ and $\mathbb{I}[G](m_1,m_2,\ldots,m_n)\cong G[\overline{K_m}]$.

A game coloring of a graph is a coloring of the vertices in which two players Ann (first player) and Ben are alternatively coloring the vertices of the graph $G$ properly by using a fixed set of colors $C$. 
The first player Ann is aiming to get a proper coloring of the whole graph, where as the second player Ben is trying to prevent the realization of this project. 
If all the vertices are colored then Ann wins the game, otherwise Ben wins (that is, at that stage of the game there appears a block vertex. 
A $block$ vertex means an uncolored vertex which has all colors from $C$ on its neighbors). 
The minimum number of colors required for Ann to win the game on a graph $G$ irrespective of Ben's strategy is called the game chromatic number of the graph $G$ and it is denoted by $\chi_g(G)$. 
There has been a lot of papers on  game coloring. See for instance, \cite{gua,sek,wu,zhu}. 
The idea of indicated coloring was introduced by A. Grzesik  in \cite{and} as a slight variant of the game coloring in the following way: in each round the first player Ann selects a vertex and then the second player Ben colors it properly, using a fixed set of colors. 
The aim of Ann as in game coloring is to achieve a proper coloring of the whole graph $G$, while Ben tries to ``block'' some vertex.
The smallest number of colors required for Ann to win the game on a graph $G$ is known as the indicated chromatic number of $G$ and is denoted by $\chi_i(G)$.
Clearly from the definition we see that $\omega(G)\leq\chi(G)\leq\chi_i(G)\leq\Delta(G)+1$. For a graph $G$, if Ann has a winning strategy using $k$ colors, then we say that $G$ is $k$-indicated colorable.

%In \cite{and}, A. Grzesik has showed that there are examples of graphs for which $\chi(G)<\chi_i(G)$ and $4\chi(G)$ is a rough upper bound for the indicated chromatic number of random graphs. Also he has investigated several families of graphs for which $\chi_i(G)=\chi(G)$, and showed that if $G\ (\ \neq K_n )$ is a uniquely colorable perfect graph with $\chi_i(G)=\chi(G)$, then it contains a clique-pair of size $\omega(G)$. Further in \cite{pan},  R. Pandiyaraj et.al. have discussed the Brooks' type results for the indicated coloring and have showed that there are regular graphs  $G$ other than $K_n$ and $C_{2k+1}$ for which $\chi_i(G)=\Delta(G)+1$. Also they have showed that $\chi_i(G)\leq\mathrm{col}(G)$, for any graph $G$.

A. Grzesik in \cite{and} has shown that a graph $G$ being $k$-indicated colorable need not naturally guarantee that $G$ is $(k+1)$-indicated colorable. Thus A. Grzesik raised the following question: For a graph $G$, if $G$ is $k$-indicated colorable, will it imply that $G$ is also $(k+1)$-indicated colorable? The question still remains open. One can equivalently characterize all graphs $G$ which are $k$-indicated colorable for all $k\geq\chi_i(G)$.
There has been already some partial answers to this question. 
For instance in \cite{ff,ffg,fran,pan}, it has been proved that the chordal graphs, cographs, complement of bipartite graphs, $\{P_5,K_3\}$-free graphs, $\{P_5, paw\}$-free graphs, $\{P_5,K_4-e\}$-free graphs, $\{P_5,K_4,Kite,Bull\}$-free graphs, connected $\{P_6,C_5,\overline{P_5}, K_{1,3}\}$-free graphs which contain an induced $C_6$, $\mathbb{K}[C_5](m_1,m_2,\ldots,m_5)$, $\{P_2\cup P_3, C_4\}$-free graphs, connected $\{P_5,\overline{P_2\cup P_3},\overline{P_5},Dart\}$-free graphs which contain an induced $C_5$, $\{P_5,C_4\}$-free graphs, %if $T$ be any tree, then 
for $n,m\geq3$, $K_m\Box T$, $C_n\Box T$, $K_m\Box C_n$ and $G^*\Box T$ are $k$-indicated colorable for all $k$ greater than or equal to their chromatic numbers, where $T$ is any tree and $G^*$ is a $\{2K_2,C_4\}$-free graph.
%In addition, P. Francis and S. Francis Raj has proved in \cite{ff} that if $T$ be any tree, then for $n,m\geq3$, $K_m\Box T$, $C_n\Box T$, $K_m\Box C_n$ and $\{2K_2,C_4\}$-free graph $\Box T$ are $k$-indicated colorable for all $k$ greater than or equal to the indicated 
In addition, M. Laso\'{n} in \cite{las} has obtained the indicated chromatic number of matroids. 
In this paper, we have shown that for any graphs $G$ and $H$, $G[H]$ is $k$-indicated colorable for all $k\geq \mathrm{col}(G)\mathrm{col}(H)$. 
Also, we try to add some more families of graphs which are $k$-indicated colorable for all $k\geq\chi(G)$.

In \cite{gell}, D. P. Geller and S. Stahl have proved that for any graphs $G$ and $H$, if $\chi(H)=\ell$, then $\chi(G[H])=\chi(G[K_\ell])$.
We have proved a similar type of result for indicated coloring.
For any graph $G$ and for some special families of graphs $H$ with $\chi(H)=\chi_i(H)=\ell$, we have shown that $G[H]$ is $k$-indicated colorable if and only if $G[K_\ell]$ is $k$-indicated colorable.
%This result asserts that to study the indicated coloring of the lexicographic product of two graphs $G$ and $H$ where $H$ obeys the condition stated above, it is enough to consider the second graph as a complete graph.
One can observe that the lexicographic product of a graph $G$ with a complete graph is a particular case of  the complete expansion of the graph $G$.
In this direction, we are interested in studying the indicated coloring of the complete expansion of a few families of graphs.
In this paper, we have proved that if $T$ is a tree on at least 3 vertices and if $G$ is a graph which is $T$-free or $C_\ell$-free, $\ell\geq4$ (or) $\overline{P_t}$-free, $t\geq4$, then the complete expansion of $G$ is also $T$-free or $C_\ell$-free (or) $\overline{P_t}$-free respectively.
Also we have proved that if the graph $G\cong\mathbb{K}[H]$, where $H\in\mathcal{G}=\Big\{$Chordal graphs, Cographs, Complement of bipartite graphs, $\{P_5,C_4\}$-free graphs, connected $\{P_6,C_5,\overline{P_5}, K_{1,3}\}$-free graphs which contain an induced $C_6$, Complete multipartite graphs$\Big\}$, 
%is the complete expansion of complement of bipartite graphs, chordal graphs, $P_4$-free graphs, $\{P_5,C_4\}$-free graphs, $\{P_2\cup P_3, C_4\}$-free graphs, connected $\{P_6, C_5, K_{1,3},\overline{P_5}\}$-free graphs which contain an induced $C_6$, independent expansion of $C_5$ and complete multipartite graphs, 
then $G$ is $k$-indicated colorable for all $k\geq\chi(G)$. 
As a consequence of these results, we have shown that  if $G\in\mathcal{G}$ and $H\in\mathcal{F}=\Big\{$Bipartite graphs, Chordal graphs, Cographs, $\{P_5,K_3\}$-free graphs, $\{P_5,Paw\}$-free graphs, Complement of bipartite graphs, $\{P_5,K_4,Kite,Bull\}$-free graphs, connected $\{P_6,C_5,\overline{P_5},K_{1,3}\}$-free graphs which contain an induced $C_6$, $\mathbb{K}[C_5](m_1, m_2 ,\ldots,m_5)$, $\{P_5,C_4\}$-free graphs, connected $\{P_5,\overline{P_2\cup P_3},\overline{P_5},Dart\}$-free graphs which contain an induced $C_5\Big\}$, then $G[H]$ is $k$-indicated colorable for every $k\geq \chi(G[H])$. 
%This serves as a partial answer to one of the questions raised by A. Grzesik in \cite{and}.
In addition, if $G$ is a Bipartite graph or a $\{P_5,K_3\}$-free graph (or) a $\{P_5,Paw\}$-free graph and $H\in\mathcal{F}$, then we have shown that $\chi_i(G[H])=\chi(G[H])$.

Notations and terminologies not mentioned here are as in \cite{west}.
%%%%%%%%%%%%%%%%%%%%%%%%%%%%%%%%%%%%%%%%%%%%%%%%%%%%%%%%%%%%%%%%%%%%%%%%%%%%%%%%%%%%%%%%%%%%%%%%%
%%%%%%%%%%%%%%%%%%%%%%%%%%%%%%%%%%%%%%%%%%%%%%%%%%%%%%%%%%%%%%%%%%%%%%%%%%%%%%%%%%%%%%%%%%%%%%%%%
\section{Indicated coloring of lexicographic product of graphs}\label{lexi}
%Also if $S$ and $T$ are independent sets in $G$ and $H$ respectively, then $ S\square T$ forms an independent set in $G\square H$. Cartesian product is one of the widely investigated product among the several graph products. The structure of the Cartesian product can be easily visualized. A typical example is hypercubes ($Q_n\cong K_2\square K_2\square\ldots\square K_2 \{n $ times$\})$.
Let us start Section \ref{lexi} by recalling a result proved in \cite{pan}.
%%%%%%%%%%%%%%%%%%%%%%%%%%%%%%%%%%%%%%%%%%%%%%%%%%%%%%%%%%%%%%%%%%%%%%%%%%%%%%%%%%%%%%%%%%%%%%%%%
\begin{thm}[\cite{pan}]\label{pan}
\label{col}Any graph $G$ is $k$-indicated colorable for all $k\geq\mathrm{col}(G)$.
\end{thm}
%%%%%%%%%%%%%%%%%%%%%%%%%%%%%%%%%%%%%%%%%%%%%%%%%%%%%%%%%%%%%%%%%%%%%%%%%%%%%%%%%%%%%%%%%%%%%%%%%
Let us find a result relating the indicated coloring of $G[H]$ and the coloring number of $G$ and $H$.
%%%%%%%%%%%%%%%%%%%%%%%%%%%%%%%%%%%%%%%%%%%%%%%%%%%%%%%%%%%%%%%%%%%%%%%%%%%%%%%%%%%%%%%%%%%%%%%%%
\begin{thm}
\label{colbound}For any graphs $G$ and $H$, $G[H]$ is $k$-indicated colorable for all $k\geq\mathrm{col}(G)\mathrm{col}(H)$.
\end{thm}
\begin{proof}
Let $G$ and $H$ be any graphs with $n$ and $n'$ vertices respectively.
By the definition of $\mathrm{col}(G)$ and $\mathrm{col}(H)$, the vertices of $G$ and $H$ can be ordered as $u_1,u_2,\ldots,u_n$ and $v_1,v_2,\ldots,v_{n'}$, such that for $1\leq i\leq n$, $d(u_i)<\mathrm{col}(G)$ in $\langle u_1,u_2,\ldots,u_i\rangle$ and for $1\leq j\leq n'$, $d(v_j)<\mathrm{col}(H)$ in $\langle v_1,v_2,\ldots,v_{j}\rangle$ respectively.
%Let the vertices of $G$ and $H$ be arranged in the order $u_1,u_2,\ldots,u_n$ and $v_1,v_2,\ldots,v_{n'}$ respectively, such that $d(u_1)\geq d(u_2)\geq\ldots\geq d(u_n)$ and $d(v_1)\geq d(v_2)\geq\ldots\geq d(v_{n'})$. 
%We know that, $\mathrm{col}(G)=1+\min\limits_{1\leq i\leq n}\{\max\{d(u_i),i-1\}\}$ and $\mathrm{col}(H)=1+\min\limits_{1\leq i\leq n'}\{\max\{d(v_i),i-1\}\}$. 
Also it can be seen that by presenting the vertices in the order  $u_1,u_2,\ldots,u_n$, Ann has a winning strategy using $\ell$ colors for $G$, for every $\ell\geq$ col$(G)$. The same can be observed for $H$. 

Let $k\geq\mathrm{col}(G)\mathrm{col}(H)$. Let Ann start by presenting the vertices of $H_{u_1}$ in the order $(u_1,v_1), \linebreak (u_1,v_2),\ldots,(u_1,v_{n'})$ until Ben uses $\mathrm{col}(H)$ colors and let $(u_1,v_{j_1})$, $j_1\leq n'$, be the last vertex presented by Ann in $H_{u_1}$.  
%By definiton of $\mathrm{col}(H)$, Ben has available colors to color the vertices of $H_{u_1}$.
Next, let Ann present the vertices of $H_{u_2}$ in the  order $(u_2,v_1),  (u_2,v_2),\ldots,(u_2,v_{n'})$ until Ben uses $\mathrm{col}(H)$ colors in $H_{u_2}$ and let $(u_2,v_{j_2})$, $j_2\leq n'$, be the last vertex presented by Ann in $H_{u_2}$. 
Let us assume that Ann has followed the same strategy to present the vertices of $H_{u_3},\ldots,H_{u_{i-1}}$, $i\leq n$ until Ben uses  $\mathrm{col}(H)$ colors in each of $H_{u_p}$, $3\leq p\leq i-1$ and let $(u_p,v_{j_p})$, $j_p\leq n'$, be the last vertex presented by Ann in $H_{u_p}$. 
While considering the vertices in $H_{u_i}$, each of them are adjacent to at most $(\mathrm{col}(G)-1)\mathrm{col}(H)$  distinct colors given to the vertices in $\cup_{j=1}^{i-1} H_{u_j}$. 
So Ben has at least $k-(\mathrm{col}(G)-1)\mathrm{col}(H)\geq \mathrm{col}(G)\mathrm{col}(H)-(\mathrm{col}(G)-1)\mathrm{col}(H)=\mathrm{col}(H)$ colors available for the vertices of $H_{u_i}$. 
This is true for every $i\leq n$. 
Thus Ann can follow the same strategy to present the vertices of $H_{u_i}$, $3\leq i\leq n$ until Ben uses $\mathrm{col}(H)$ colors and let $(u_i,v_{j_i})$, $j_i\leq n'$, be the last vertex presented by Ann in $H_{u_i}$.%   and for $3\leq i\leq n$, let $u_i,v{j_i}$ be the last vertex presented by Ann in $H_{u_i}$. In this strategy, Ben can never create a block vertex for the following reason.  Ifor present the vertices of 
 
%While coloring, for $2\leq i\leq n$, if Ann presented some of the vertices of $H_{u_1},H_{u_2},\ldots,H_{u_{i-1}}$, then every vertices in $H_{u_i}$ are adjacent to at most $(\mathrm{col}(G)-1)\mathrm{col}(H)$ colored neighbors. So for $2\leq i\leq n$, Ben has at least $k-(\mathrm{col}(G)-1)\mathrm{col}(H)\geq \mathrm{col}(G)\mathrm{col}(H)-(\mathrm{col}(G)-1)\mathrm{col}(H)=\mathrm{col}(H)$ available colors to color the vertices of $H_{u_i}$. 

Finally, let Ann present the remaining vertices of $G[H]$ in the order $(u_1,v_{j_1+1}),\linebreak(u_1,v_{j_1+2}),\ldots,(u_1,v_{n'}),(u_2,v_{j_2+1}),(u_2,v_{j_2+2}),\ldots,(u_2,v_{n'}),\ldots,(u_n,v_{{j_n}+1}),(u_n,v_{j_n+2}),\ldots,(u_n,v_{n'})$. For $1\leq i\leq n$, since the neighbors outside $H_{u_i}$ for any vertex in $H_{u_i}$ is the same, the $\mathrm{col}(H)$ colors given to the vertices in $H_{u_i}$ are always available to the uncolored vertices in $H_{u_i}$ and hence Ben cannot create a block vertex. Thus Ann has a winning strategy using $k$ colors, for every $k\geq$ col$(G)$col$(H)$.
\end{proof}
%%%%%%%%%%%%%%%%%%%%%%%%%%%%%%%%%%%%%%%%%%%%%%%%%%%%%%%%%%%%%%%%%%%%%%%%%%%%%%%%%%%%%%%%%%%%%%%%%
By Theorem \ref{pan}, we know that $\mathrm{col}(G[H])$ is an upper bound for $\chi_i(G[H])$. To exhibit that $\mathrm{col}(G)\mathrm{col}(H)$ is a better upper bound for $\chi_i(G[H])$, let us show that $\mathrm{col}(G[H])-\mathrm{col}(G)\mathrm{col}(H)$ can be arbitrarily large. 
%%%%%%%%%%%%%%%%%%%%%%%%%%%%%%%%%%%%%%%%%%%%%%%%%%%%%%%%%%%%%%%%%%%%%%%%%%%%%%%%%%%%%%%%%%%%%%%%%
%\begin{thm}
%\label{bound}There exist graphs $G$ and $H$ such that $\mathrm{col}(G[H])-\mathrm{col}(G)\mathrm{col}(H)$ is arbitrarily large.
%\end{thm}
%\begin{proof}
Let $G$ and $H$ be any two graphs and let $G'\subseteq G$ and $H'\subseteq H$ such that $\delta(G')=\mathrm{col}(G)-1$ and $\delta(H')=\mathrm{col}(H)-1$. 
It can be easily observed that $\delta(G[H])=\delta(G)|H|+\delta(H)$. Hence 

$\begin{array}{rl}
\mathrm{col}(G[H])\geq&\mathrm{col}(G'[H'])\\
\geq&1+\delta(G'[H'])\\
=&1+\delta(G')|H'|+\delta(H')\\
=&1+(\mathrm{col}(G)-1)|H'|+\mathrm{col}(H)-1\\
=&\mathrm{col}(G)|H'|-|H'|+\mathrm{col}(H)\\
\end{array}$

\noindent Therefore, we see that\\  %Subtracting $\mathrm{col}(G)\mathrm{col}(H)$ on both sides\\
$\begin{array}{rl}
\mathrm{col}(G[H])-\mathrm{col}(G)\mathrm{col}(H)\geq&\mathrm{col}(G)|H'|-\mathrm{col}(G)\mathrm{col}(H)-|H'|+\mathrm{col}(H)\\
=&\mathrm{col}(G)(|H'|-\mathrm{col}(H))-\left(|H'|-\mathrm{col}(H)\right)\\
=&(\mathrm{col}(G)-1)(|H'|-\mathrm{col}(H))\\
=&(\mathrm{col}(G)-1)(|H'|-1-\delta(H'))
\end{array}$

For every non-complete graph $H'$, $|H'|-1-\delta(H')$ is strictly positive. So we can suitably choose $G$ and $H$ such that $(\mathrm{col}(G)-1)(|H'|-1-\delta(H'))$ is arbitrarily large.

%\end{proof}
%%%%%%%%%%%%%%%%%%%%%%%%%%%%%%%%%%%%%%%%%%%%%%%%%%%%%%%%%%%%%%%%%%%%%%%%%%%%%%%%%%%%%%%%%%%%%%%%%
%We discussed in Section \ref{intro} that increasing the number of colors will not favor for Ann.
%So, if Ann started following a winning strategy using some $k$ colors, then there is no guarantee for Ann to win if Ann uses $k'$ colors, for $k'> k$. 
We now define a family $\mathcal{H}$ of graphs. 
%A graph $G$ belongs to $\mathcal{H}$ if the winning strategy of Ann for $G$ is independent of the choice of the number of colors used. 

 A graph $G$ belongs to $\mathcal{H}$ if Ann has a winning strategy using $\chi_i(G)$ colors which she can follow until Ben uses $\chi_i(G)$ colors for the vertices of $G$ and for the remaining vertices she has a way of extending this to a winning strategy using $k$ colors, for any $k\geq\chi_i(G)$. %from the stage when all the $\chi_i(G)$ colors are used, can be extended to a winning strategy using $k$ colors, for any $k\geq\chi_i(G)$. %following condition is satisfied. But if there exists a graph $G$ such that Ann has a common winning strategy using any $k$ colors, for $k\geq\chi_i(G)$, then Ann has a strategy for $G$ to color the vertices until $\chi_i(G)$ colors given by Ben and then Ann has a strategy to continue coloring using any $k$ colors, where $k\geq\chi_i(G)$.
%Let us denote such graphs as $\mathcal{H}$ in this paper. 

Let us now consider the indicated coloring of the lexicographic product of any graph $G$ with a graph $H\in \mathcal{H}$ with $\chi(H)=\chi_i(H)$.
%By Theorem \ref{HwithKl}, for the graphs $H\in\mathcal{H}$, it is enough to find the indicated chromatic number of $G[K_l]$ instead of finding the indicated chromatic number of $G[H]$. 
%%%%%%%%%%%%%%%%%%%%%%%%%%%%%%%%%%%%%%%%%%%%%%%%%%%%%%%%%%%%%%%%%%%%%%%%%%%%%%%%%%%%%%%%%%%%%%%%%
\begin{thm}
\label{HwithKl}For any graph $G$ and for any graph $H\in \mathcal{H}$ with $\chi(H)=\chi_i(H)=\ell$, %such that Ann has a strategy for $H$ to color the vertices until $l$ colors given by Ben and then Ann has a strategy to continue coloring using any $l'$ colors, where $l'\geq l$, then 
$G[H]$ is $k$-indicated colorable if and only if $G[K_\ell]$ is $k$-indicated colorable.
In particular, $\chi_i(G[H])=\chi_i(G[K_\ell])$.
\end{thm}
\begin{proof}
Let $G$ be any graph and $H\in \mathcal{H}$ be a graph with $\chi(H)=\chi_i(H)=\ell$ whose vertices are $u_1,u_2,\ldots,u_n$ and $v_1,v_2,\ldots,v_{n'}$ respectively. 
Let us first assume that $G[K_\ell]$ is $k$-indicated colorable
and let  $st_{G[K_\ell]}$ denote a  winning strategy of Ann for $G[K_\ell]$ using $k$ colors. 
Also, let $st_H$ be a winning strategy of Ann for $H$ using $\ell$ colors. Corresponding to the strategy $st_H$ of $H$, for  $1\leq i\leq n$, Ann can get a winning strategy for $H_{u_i}$, by presenting the vertex $(u_i,v)$ whenever $v$ is presented in the strategy $st_H$. Let us call this winning strategy of $H_{u_i}$ as $st_{H_i}$.   
Using the strategies $st_{G[K_\ell]}$ and $st_{H_i}$, for $1\leq i\leq n$, we shall construct a winning strategy for Ann for the graph $G[H]$ using $k$ colors as follows.% to color the vertices of $G[H]$. 

In $st_{G[K_\ell]}$, if the first vertex presented by Ann belongs to $K_{\ell_{u_i}}$, for some $i$, $1\leq i\leq n$, then let Ann present the first vertex  from $H_{u_i}$ by following the strategy $st_{H_i}$ of $H_{u_i}$.
If Ben colors it with a color, say $c_1$, then we continue with the strategy $st_{G[K_\ell]}$ by assuming that the color $c_1$ is given to the vertex which was presented in $K_{\ell_{u_i}}$.
If the second vertex presented by Ann in the strategy $st_{G[K_\ell]}$ belongs to $K_{\ell_{u_j}}$, for some $j$ (not necessarily distinct from $i$), $1\leq j\leq n$, then as per the strategy of $st_{H_j}$, let Ann present the vertices  of $H_{u_j}$  until a new color is given by Ben to a vertex in $H_{u_j}$, say $c_2$. 
That is, if Ann presents the vertices from the same $H_{u_i}$, then Ann  will continue presenting the vertices until a vertex from a new color class in $H_{u_i}$ is presented. 
Instead, if Ann presents a vertex from  $H_{u_j}$, $i\neq j$ and $1\leq j\leq n$, then that vertex will be a vertex from a new color class in $H_{u_j}$. 
This is because this is the first vertex presented from $H_{u_j}$. 
Then we continue with the strategy $st_{G[K_\ell]}$ by assuming that the color $c_2$ is given to the vertex presented  by Ann in $K_{\ell_{u_j}}$. 
In general,  if the vertex presented by Ann in $st_{G[K_\ell]}$ belongs to $K_{\ell_{u_r}}$, for some $r$, $1\leq r\leq n$,  then in $G[H]$, let Ann present the vertices  in $H_{u_r}$ by continuing with the strategy $st_{H_r}$,  until a new color is given by Ben to a vertex in $H_{u_r}$, say $c_r$. 
Now we shall continue with the strategy $st_{G[K_\ell]}$ by assuming that the color $c_r$ is given to the vertex presented by Ann in $K_{\ell_{u_r}}$. Repeat this process until all the vertices in $G[K_\ell]$ have been presented  using the strategy $st_{G[K_\ell]}$. While following this strategy in $G[H]$,  suppose for some $p,q$, $1\leq p\leq n$, $1\leq q\leq n'$, Ben creates a block vertex $(u_p,v_q)$ in $G[H]$. Then $(u_p,v_q)$ must be adjacent to  all the $k$ colors. 
According to the Ann strategy for $G[H]$, if a vertex of $G[H]$ in $H_{u_p}$ is adjacent with a color, then there exists a vertex of $G[K_\ell]$ in $K_{\ell_{u_p}}$ which is adjacent with the same color. 
Thereby, there exists an uncolored vertex of $G[K_\ell]$ in $K_{\ell_{u_p}}$ which is a block vertex, a contradiction to $st_{G[K_{\ell}]}$ being a winning strategy of $G[K_\ell]$. 
So, Ben cannot create a block vertex in $G[H]$ when Ann follows this strategy.

At this stage, that is, when all the vertices in $G[K_\ell]$ have been presented using the strategy $st_{G[K_\ell]}$ as shown above, we see that the number of  %it can be seen that since $\chi(H)=\ell$, for every $i$, $1\leq i\leq n$, the number of 
colors used in $H_{u_i}$, for $1\leq i\leq n$, will be exactly $\ell$. Also %Also while coloring the vertices of $H_{u_i}$, we are following the strategy $st_H$, so Ben can color it with at most $\ell$ colors.
 there maybe some uncolored vertices left in $G[H]$. 
For $1\leq i,j\leq n$, the colors given to the vertices of  $H_{u_i}$ cannot be given to the vertices of $H_{u_j}$, for any $u_j$ such that $u_iu_j\in E(G)$. Thus Ben has at least $\ell$ colors available to color the remaining uncolored vertices of $H_{u_i}$. Also by our assumption that $H\in\mathcal{H}$, even if the number of colors available for Ben is $\ell'\geq \ell$,  Ann will still have a winning strategy to present the remaining uncolored vertices of $H_{u_i}$ using $\ell'$ colors.
Hence $G[H]$ is $k$-indicated colorable.

Now, let us assume that $G[H]$ is $k$-indicated colorable.   % for any strategy followed by Ben. One can see that, in any winning strategy of Ann, for $1\leq i\leq n$, every copies of $H_{u_i}$ has at least $\ell$ available colors. 
%And if we follow similar strategy as earlier to the vertices of $G[K_\ell]$, every copies of $K_{\ell_{u_i}}$ has at least $\ell$ available colors. 
%And there are only $\ell$ vertices in $K_\ell$, so Ann can win the game using $q$ colors. 
%So here is the similar strategy Ann follows.
Let $f$ be some $\chi$-coloring of $H$. Corresponding to this $\chi$-partition of $H$, for any $i$ such that $1\leq i\leq n$, we can get a $\chi$-partition for $H_{u_i}$, say $f_i$, by placing two vertices $(u_i,v)$ and $(u_i,w)$ in the same color class whenever $v$ and $w$ are in the same color class in $f$. Also, let $st_{G[H]}$ be a  winning strategy of Ann  using $k$ colors for $G[H]$ when Ben uses the following strategy: For each of the $H_{u_i}$, Ben will not change the color classes of $f_i$. That is, Ben will color two vertices with the same color in $H_{u_i}$ if and only if they belong to the same color class in $f_i$.  When we say that we use the strategy $st_{G[H]}$ for $G[H]$, it means that the strategy followed by Ben will be the fixed strategy  mentioned in the previous line. Using this strategy $st_{G[H]}$ and the $\chi$-partitions $f_i$, we shall construct a winning strategy for Ann for the graph $G[K_\ell]$ using $k$ colors as follows.

In $st_{G[H]}$, if the first vertex presented by Ann belongs to  $H_{u_i}$, for some $i$, $1\leq i\leq n$, then let Ann present the first vertex from $K_{\ell_{u_i}}$.
If Ben colors it with a color, say $c_1$, then we continue with the strategy $st_{G[H]}$ by assuming that the color $c_1$ is given to the vertex which was presented in $H_{u_i}$. 
As per $st_{G[H]}$, let Ann continue by presenting the vertices of $G[H]$ until a vertex of a new color class (with respect to the coloring $f_j$) in some $H_{u_j}$ (not necessarily different from $H_{u_i}$), $1\leq j\leq n$, is presented by Ann.  That is, if Ann presents the vertices from the same $H_{u_i}$, then Ann  will continue presenting the vertices until a vertex from a new color class (with respect to the coloring $f_i$) in $H_{u_i}$ is presented. 
Instead, if Ann presents a vertex from a $H_{u_j}$, $i\neq j$ and $1\leq j\leq n$, then that vertex will be a vertex from a new color class in $H_{u_j}$. 
This is because this is the first vertex presented from $H_{u_j}$.   
%Next, as per the strategy $st_{G[H]}$, let Ann present the next vertex (vertices) from $H_{u_j}$, for some $j$(not necessarily distinct), $1\leq j\leq n$, until  then let Ann present a vertex from a new color class according to $d$.
Then in $G[K_\ell]$, let Ann present the next vertex in $K_{\ell_{u_j}}$. 
If Ben colors it with the color, say $c_2$, then we continue with the strategy $st_{G[H]}$ by assuming that the color $c_2$ is given to the vertex from the new color class in $H_{u_j}$.  
Again, as per $st_{G[H]}$, let Ann continue by presenting the vertices of $G[H]$ until a vertex of a new color class (with respect to the coloring  $f_s$) in some $H_{u_s}$, $1\leq s\leq n$, is presented by Ann. 
Continue this strategy until all the vertices of $G[K_\ell]$ are presented.  
While following this strategy, 
%In general, if the vertex presented by Ann as per $st_{G[H]}$ belongs to $H_{u_i}$, for some $i$, $1\leq i\leq n$, such that it is received a new color according to $d$, then in $G[K_\ell]$, let Ann present a new vertex in $K_{\ell_{u_i}}$. 
%If Ben colors it with the color, say $c_i$. 
%And then we continue with the strategy $st_{G[H]}$ assuming the color $c_i$ is given to the vertex presented by Ann in $H_{u_i}$.
%Suppose in $G[H]$ if Ann presented a vertex in $H_{u_i}$, such that it is received a color already assigned, say $c_i'$, according to $d$, then we continue with the strategy $st_{G[H]}$ assuming the color $c_i'$ is given to the vertex presented by Ann in $H_{u_i}$.
suppose for some $p$, $1\leq p\leq n$, Ben creates a block vertex $(u_p,w)$ in $K_{\ell_{u_p}}$. 
Then $(u_p,w)$ will be adjacent with vertices receiving all the $k$ colors. 
 Since there can be at most $\ell-1$ colored vertices in $K_{\ell_{u_p}}$, $(u_p,w)$ must have neighbors with at least $k-\ell+1$ distinct colors outside $K_{\ell_{u_p}}$.
Therefore every vertex in $H_{u_p}$ will be adjacent with the vertices receiving at least $k-\ell+1$ distinct colors outside $H_{u_p}$. 
Thus the number of available colors for $H_{u_p}$ in $G[H]$ is at most $\ell-1$, a contradiction to $st_{G[H]}$ being a winning strategy for $G[H]$.
So, Ben cannot create a block vertex in $G[K_\ell]$.
Hence  $G[K_\ell]$ is $k$-indicated colorable.
\end{proof}

\section{Consequences of Theorem \ref{HwithKl}}\label{conseq}

Let us recall some of the results shown in \cite{ffg}, \cite{and} and \cite{pan}.

%%%%%%%%%%%%%%%%%%%%%%%%%%%%%%%%%%%%%%%%%%%%%%%%%%%%%%%%%%%%%%%%%%%%%%%%%%%%%%%%%%%%%%%%%%%%%%%%%
\begin{thm}[\cite{pan}]\label{union}
Let $G=G_1\cup G_2$. If $G_1$ is $k_1$-indicated colorable for every
$k_1\geq \chi_i(G_1)$ and $G_2$ is $k_2$-indicated colorable for every $k_2 \geq \chi_i(G_2)$, then
$\chi_i(G)=\max\{\chi_i(G_1),  \chi_i(G_2)\}$ and $G$ is $k$-indicated colorable for every $k\geq\chi_i(G)$.
\end{thm}
%%%%%%%%%%%%%%%%%%%%%%%%%%%%%%%%%%%%%%%%%%%%%%%%%%%%%%%%%%%%%%%%%%%%%%%%%%%%%%%%%%%%%%%%%%%%%%%%%

\begin{thm}[\cite{ffg,and,pan}]\label{bipart}
Let $\mathcal{F} = \Big\{$Bipartite graphs, Chordal graphs, Cographs, $\{P_5,K_3\}$-free graphs, $\{P_5,Paw\}$-free graphs, Complement of bipartite graphs, $\{P_5,K_4,Kite,Bull\}$-free graphs, connected $\{P_6,C_5,\overline{P_5},K_{1,3}\}$-free graphs which contain an induced $C_6$, $\mathbb{K}[C_5](m_1, m_2 ,\ldots,m_5)$, $\{P_5,C_4\}$-free graphs, connected $\{P_5,\overline{P_2\cup P_3},\overline{P_5},Dart\}$-free graphs which contain an induced $C_5\Big\}$. If $G\in \mathcal{F}$, then $G$ is $k$-indicated colorable for all $k\geq \chi(G)$.
\end{thm}

In \cite{ffg,pan,and}, if one closely observe the proof's of the families of graphs in $\mathcal{F}$ while showing that they are $k$-indicated colorable for every $k\geq \chi(G)$, we can see that the winning strategy of Ann will be independent of the choice of $k$. Hence any graph in $\mathcal{F}$ is also a graph in $\mathcal{H}$.

\begin{thm}\label{fch}
\ $\mathcal{F}\subseteq \mathcal{H}$.
\end{thm}

%%%%%%%%%%%%%%%%%%%%%%%%%%%%%%%%%%%%%%%%%%%%%%%%%%%%%%%%%%%%%%%%%%%%%%%%%%%%%%%%%%%%%%%%%%%%%%%%%
%%%%%%%%%%%%%%%%%%%%%%%%%%%%%%%%%%%%%%%%%%%%%%%%%%%%%%%%%%%%%%%%%%%%%%%%%%%%%%%%%%%%%%%%%%%%%%%%%
%%%%%%%%%%%%%%%%%%%%%%%%%%%%%%%%%%%%%%%%%%%%%%%%%%%%%%%%%%%%%%%%%%%%%%%%%%%%%%%%%%%%%%%%%%%%%%%%%
%%%%%%%%%%%%%%%%%%%%%%%%%%%%%%%%%%%%%%%%%%%%%%%%%%%%%%%%%%%%%%%%%%%%%%%%%%%%%%%%%%%%%%%%%%%%%%%%%

%So in order to find $\chi_i(G[H])$, it is enough to find $\chi_i(\mathbb{K}[G](\chi(H),\chi(H),\ldots,\chi(H)))$, if $H\in\mathcal{H}$.
%Also if $\mathbb{K}[G]$ is $k$-indicated colorable, then $G[H]$ is also $k$-indicated colorable. 
% for the indicated coloring of the independent expansion of any graph $G$  and complete expansin of $C_5$.
%
%\begin{thm}[\cite{ffg}]\label{indexp}
%%For  $1\leq i\leq n$, let $m_i$'s be  positive integers.
%If $G$ is a graph which is $k$-indicated colorable for all $k\geq\chi_i(G)$ then the graph $\mathbb{I}[G]$ is also $k$-indicated colorable for all $k\geq\chi_i(G)$.
%\end{thm}
%%%%%%%%%%%%%%%%%%%%%%%%%%%%%%%%%%%%%%%%%%%%%%%%%%%%%%%%%%%%%%%%%%%%%%%%%%%%%%%%%%%%%%%%%%%%%%%%%
%\begin{thm}[\cite{ffg}]\label{comexp}
%%For  $1\leq i\leq 5$, let $m_i$'s be  positive integers.
%If $G$ is either  $\mathbb{K}[C_5]$ or $\{P_5,C_4\}$-free graph then $G$ is $k$-indicated colorable for all $k\geq\chi(G)$. Moreover that $\chi(\mathbb{K}[C_5])=\max\left\{\omega(\mathbb{K}[C_5]),\left\lceil\frac{|V(\mathbb{K}[C_5])|}{2}\right\rceil\right\}$.
%\end{thm}
%%%%%%%%%%%%%%%%%%%%%%%%%%%%%%%%%%%%%%%%%%%%%%%%%%%%%%%%%%%%%%%%%%%%%%%%%%%%%%%%%%%%%%%%%%%%%%%%%
\begin{thm}[\cite{ffg}]\label{comexpbi}
Let $G$ be a bipartite graph and $G'\cong\mathbb{K}[G](m,m,\ldots,m)\cong G[K_m]$ be the complete expansion of $G$, for some $m\geq1$. Then $\chi_i(G')=2m=\chi(G')$.
\end{thm}
%%%%%%%%%%%%%%%%%%%%%%%%%%%%%%%%%%%%%%%%%%%%%%%%%%%%%%%%%%%%%%%%%%%%%%%%%%%%%%%%%%%%%%%%%%%%%%%%%
%Let us recall some of the results shown in  \cite{pan}.
%We know that, for the union of two graphs $G_1$ and $G_2$, $\chi(G_1\cup G_2)=\max\{\chi(G_1),\chi(G_2)\}$. The same holds even for the indicated chromatic number.
%%%%%%%%%%%%%%%%%%%%%%%%%%%%%%%%%%%%%%%%%%%%%%%%%%%%%%%%%%%%%%%%%%%%%%%%%%%%%%%%%%%%%%%%%%%%%%%%%
 Corollary \ref{bipartproduct} is an immediate consequence of Theorem \ref{HwithKl} and Theorem \ref{comexpbi}.
%%%%%%%%%%%%%%%%%%%%%%%%%%%%%%%%%%%%%%%%%%%%%%%%%%%%%%%%%%%%%%%%%%%%%%%%%%%%%%%%%%%%%%%%%%%%%%%%%
\begin{cor}\label{bipartproduct}
Let $G$ be a bipartite graph and $H\in\mathcal{H}$ with $\chi_i(H)=\chi(H)$, then $\chi_i(G[H])=2 \chi(H)$.
\end{cor}
%%%%%%%%%%%%%%%%%%%%%%%%%%%%%%%%%%%%%%%%%%%%%%%%%%%%%%%%%%%%%%%%%%%%%%%%%%%%%%%%%%%%%%%%%%%%%%%%%
Now let us observe the relationship between a given graph and its complement in terms of their complete expansion and independent expansion.  %expanIt is observed from the Proposition \ref{indcomexpcombi}  that the complete expansion of the complement of a bipartite graph is isomorphic to the complement of some bipartite graph.
%%%%%%%%%%%%%%%%%%%%%%%%%%%%%%%%%%%%%%%%%%%%%%%%%%%%%%%%%%%%%%%%%%%%%%%%%%%%%%%%%%%%%%%%%%%%%%%%%
\begin{obs}\label{indcomexpcombi}
Let $G$ be any graph, then $\overline{\mathbb{I}[G]}\cong\mathbb{K}[\overline{G}]$.
\end{obs}
%Let $B$ be a bipartite graph and let $G\cong\mathbb{K}[\overline{B}]$. Then $\overline{G}\cong \mathbb{I}[B]$ is a bipartite graph.

%\begin{proof}
%Let $v_1,v_2,\ldots,v_n$ be the vertices of the bipartite  graph $B$ and as in Theorem \ref{treecyclefree}, let $V_1,V_2,\ldots, V_n$ be the set of all vertices in the complete expansion of $\overline{B}$ corresponding to the vertices $v_1,v_2,\ldots,v_n$ respectively. Let $G\cong\mathbb{K}[\overline{B}]$. For $1\leq i,j,\leq n$, each $\langle V_i\rangle$ is complete in $G$, $[V_i,V_j]$ is complete in $G$ whenever $v_i$ is non-adjacent  to $v_j$ in $B$ and $[V_i,V_j]=\emptyset$ in $G$ whenever $v_i$ is adjacent  to $v_j$ in $B$.
%In $\overline{G}$,  each $\langle V_i\rangle$ is an independent set, $[V_i,V_j]=\emptyset$ whenever $v_i$ is non-adjacent  to $v_j$ in $B$ and $[V_i,V_j]$ is complete whenever $v_i$ is adjacent  to $v_j$ in $B$. Thus by the definition of independent expansion of graph, $\overline{G}\cong \mathbb{I}[B]$ is a bipartite graph.
%\end{proof}
%%%%%%%%%%%%%%%%%%%%%%%%%%%%%%%%%%%%%%%%%%%%%%%%%%%%%%%%%%%%%%%%%%%%%%%%%%%%%%%%%%%%%%%%%%%%%%%%%
 Also one can observe that if $G$ is a bipartite graph, then $\mathbb{I}[G]$ is also a bipartite graph. Thus Theorem \ref{bipart} and Observation \ref{indcomexpcombi} will yield Corollary \ref{comexpcombi}.
%%%%%%%%%%%%%%%%%%%%%%%%%%%%%%%%%%%%%%%%%%%%%%%%%%%%%%%%%%%%%%%%%%%%%%%%%%%%%%%%%%%%%%%%%%%%%%%%%
\begin{cor}\label{comexpcombi}
If $G$ is the complete expansion of the complement of a bipartite graph, then $G$ is $k$-indicated colorable for all $k\geq\chi(G)$.
\end{cor}
%%%%%%%%%%%%%%%%%%%%%%%%%%%%%%%%%%%%%%%%%%%%%%%%%%%%%%%%%%%%%%%%%%%%%%%%%%%%%%%%%%%%%%%%%%%%%%%%%
Now Corollary \ref{comexpcombiproduct} follows from Theorem \ref{HwithKl} and Corollary \ref{comexpcombi}.
%%%%%%%%%%%%%%%%%%%%%%%%%%%%%%%%%%%%%%%%%%%%%%%%%%%%%%%%%%%%%%%%%%%%%%%%%%%%%%%%%%%%%%%%%%%%%%%%%
\begin{cor}\label{comexpcombiproduct}
If $G$ is the complement of a bipartite graph and $H\in\mathcal{H}$ with $\chi_i(H)=\chi(H)$, then $G[H]$ is $k$-indicated colorable for all $k\geq\chi(G[H])$.
\end{cor}
%%%%%%%%%%%%%%%%%%%%%%%%%%%%%%%%%%%%%%%%%%%%%%%%%%%%%%%%%%%%%%%%%%%%%%%%%%%%%%%%%%%%%%%%%%%%%%%%%
Next, let us consider the complete expansion of tree-free graphs, cycle-free graphs and complement of path-free graphs.
%%%%%%%%%%%%%%%%%%%%%%%%%%%%%%%%%%%%%%%%%%%%%%%%%%%%%%%%%%%%%%%%%%%%%%%%%%%%%%%%%%%%%%%%%%%%%%%%%
\begin{prop}\label{treecyclefree}
Let $T$ be a tree on at least 3 vertices and let $G$ be a $T$-free graph or a $C_\ell$-free graph, $\ell\geq4$ (or) a $\overline{P_t}$-free graph, $t\geq4$. Then the graph $\mathbb{K}[G]$ is also  $T$-free or $C_\ell$-free (or) $\overline{P_t}$-free respectively.
\end{prop}
\begin{proof}
Let $G$ be a $C_\ell$-free  graph, for some $\ell\geq4$. 
Let $v_1,v_2,\ldots,v_n$ be the vertices of $G$ and let $V_1,V_2,\ldots, V_n$ be the set of all vertices in the complete expansion of $G$ corresponding to the vertices $v_1,v_2,\ldots,v_n$ respectively. 
Clearly $\langle V_i\rangle$, $1\leq i\leq n$ is a complete subgraph of $\mathbb{K}[G]$. 
Also if $v_iv_j\in E(G)$, then $[V_i,V_j]$  is complete in $\mathbb{K}[G]$ and any 3 vertices in $V_i\cup V_j$ will induces a $K_3$ in $\mathbb{K}[G]$. 
Let $H$ be any induced subgraph with at least 3 vertices in $\mathbb{K}[G]$. 
If $H$ contains at least 2 vertices in $V_i$ for some $i\in\{1,2,\ldots,n\}$, then $H\ncong C_\ell$. 
If $H$ contains at most one vertex in each $V_i$ then $H$ is isomorphic to  an induced subgraph of $G$. 
Thus $\mathbb{K}[G]$ is $C_\ell$-free.  
A similar proof works even for $T$-free graphs.
%where $l\geq4$.

Finally, let us consider $G$ to be a $\overline{P_t}$-free graph, for some $t\geq4$. 
Then $\overline{G}$ will be a $P_t$-free graph.
Let us consider the independent expansion of $\overline{G}$.
Let $v_1,v_2,\ldots,v_n$ be the vertices of $\overline{G}$ and let $V_1,V_2,\ldots, V_n$ be the set of all vertices in the independent expansion of $\overline{G}$ corresponding to the vertices $v_1,v_2,\ldots,v_n$ respectively. 
Let $H$ be a connected induced subgraph with exactly 4 vertices in $\mathbb{I}[\overline{G}]$. 
If $H$ contains at least 2 vertices in $V_i$ where $i\in\{1,2,\ldots,n\}$, then $H\cong C_4$ or $K_{1,3}$ (or) $K_4-e$. 
So if $H$ is any induced subgraph with at least 4 vertices in $\mathbb{I}[\overline{G}]$ and if $H$ contains at least 2 vertices in $V_i$ where $i\in\{1,2,\ldots,n\}$, then $H\ncong P_t$. 
If $H$ contains at most one vertex in each $V_i$ then $H$ is isomorphic to an induced subgraph of $\overline{G}$. 
Thus $\mathbb{I}[\overline{G}]$ is $P_t$-free. 
Therefore by Observation \ref{indcomexpcombi}, $\mathbb{K}[G]\cong\overline{\mathbb{I}[\overline{G}]}$ and hence a $\overline{P_t}$-free graph.
\end{proof}
%%%%%%%%%%%%%%%%%%%%%%%%%%%%%%%%%%%%%%%%%%%%%%%%%%%%%%%%%%%%%%%%%%%%%%%%%%%%%%%%%%%%%%%%%%%%%%%%%
As a consequence of Theorem \ref{bipart} and Proposition \ref{treecyclefree}, we obtain Corollary \ref{chorp4}.
%%%%%%%%%%%%%%%%%%%%%%%%%%%%%%%%%%%%%%%%%%%%%%%%%%%%%%%%%%%%%%%%%%%%%%%%%%%%%%%%%%%%%%%%%%%%%%%%%
\begin{cor}\label{chorp4}
If $G$ is a chordal graph or a cograph (or) a $\{P_5, C_4\}$-free graph (or) a connected $\{P_6, C_5, K_{1,3},\overline{P_5}\}$-free graph which contains an induced $C_6$, then $\mathbb{K}[G]$ is $k$-indicated colorable for all $k\geq\chi(\mathbb{K}[G])$.
\end{cor}
%%%%%%%%%%%%%%%%%%%%%%%%%%%%%%%%%%%%%%%%%%%%%%%%%%%%%%%%%%%%%%%%%%%%%%%%%%%%%%%%%%%%%%%%%%%%%%%%%
By using Theorem \ref{HwithKl} and Corollary \ref{chorp4}, we have Corollary \ref{chorp4product}.
%%%%%%%%%%%%%%%%%%%%%%%%%%%%%%%%%%%%%%%%%%%%%%%%%%%%%%%%%%%%%%%%%%%%%%%%%%%%%%%%%%%%%%%%%%%%%%%%%
\begin{cor}\label{chorp4product}
If $G$ is a chordal graph or a cograph (or) a $\{P_5, C_4\}$-free graph (or) a connected $\{P_6, C_5, K_{1,3},\overline{P_5}\}$-free graph which contains an induced $C_6$ and $H\in\mathcal{H}$ with $\chi_i(H)=\chi(H)$, then $G[H]$ is $k$-indicated colorable for all $k\geq\chi(G[H])$.
\end{cor}
%%%%%%%%%%%%%%%%%%%%%%%%%%%%%%%%%%%%%%%%%%%%%%%%%%%%%%%%%%%%%%%%%%%%%%%%%%%%%%%%%%%%%%%%%%%%%%%%%
The structural characterisation of $Paw$-free graphs and $\{P_5,K_3\}$-free graphs have been studied in \cite{ola} and \cite{sum}.
%%%%%%%%%%%%%%%%%%%%%%%%%%%%%%%%%%%%%%%%%%%%%%%%%%%%%%%%%%%%%%%%%%%%%%%%%%%%%%%%%%%%%%%%%%%%%%%%%
\begin{thm}[\cite{ola}] \label{paw}
%The graph $G$ is connected paw-free if and only if $G$ is $K_3$-free or complete multipartite.
Let $G$ be a connected graph. Then $G$ is paw-free if and only if $G$ is $K_3$-free or complete multipartite.
\end{thm}
\begin{thm}[\cite{sum}] \label{p5k3}
Every component of a $\{P_5,K_3\}$-free graph is either bipartite or $\mathbb{I}[C_5]$.
%Let $G$ be a $\{P_5,K_3\}$-free graph. Then each component of $G$ is either bipartite or $\mathbb{I}[C_5]$.
%$(m_1,m_2,\ldots,m_5)$, where $m_i\geq1$ for $ i=1,2,3,4,5$.
\end{thm}

%%%%%%%%%%%%%%%%%%%%%%%%%%%%%%%%%%%%%%%%%%%%%%%%%%%%%%%%%%%%%%%%%%%%%%%%%%%%%%%%%%%%%%%%%%%%%%%%%
For studying the indicated coloring of the lexicographic product of $\{P_5,K_3\}$-free or $\{P_5,Paw\}$-free graphs with a graph in $\mathcal{H}$ with indicated chromatic number equal to its chromatic number, let us first consider the indicated coloring of the complete expansion of the independent expansion of a graph $G$.
%%%%%%%%%%%%%%%%%%%%%%%%%%%%%%%%%%%%%%%%%%%%%%%%%%%%%%%%%%%%%%%%%%%%%%%%%%%%%%%%%%%%%%%%%%%%%%%%%
\begin{thm}\label{comexpindexpg}
  $\mathbb{K}[\mathbb{I}[G]]$ is $k$-indicated colorable if $\mathbb{K}[G]$ is $k$-indicated colorable.
\end{thm}
\begin{proof}
Let $G$ be a graph with $V(G)=\{v_1,v_2,\ldots,v_n\}$. Let  $k$ be a positive integer such that any complete expansion of $G$ is $k$-indicated colorable.
For $1\leq i\leq n$, let $V_i$ denote the independent expansion of the vertex $v_i$ in $\mathbb{I}[G]$. 
For $1\leq i\leq n$, $1\leq j\leq |V_i|$, let $H_{ij}$ be the complete subgraphs replacing the vertex $v_{ij}$ of $V_i$ in $\mathbb{K}[\mathbb{I}[G]]$ and let $H_i$ denote the $H_{ij}, 1\leq j \leq |V_i|$, with the maximum cardinality. 
Clearly 
%Clearly $\omega(\mathbb{K}[\mathbb{I}[C_5]])=\max\limits_{0\leq i\leq 4}\{|H_i|+|H_{(i+1)\mathrm{(mod\ 5)}}|\}$ and
%%By using Theorem ----, $\chi(\mathbb{K}[\mathbb{I}[C_5])=\max\left\{\omega(\mathbb{K}[\mathbb{I}[C_5]),\left\lceil\frac{\Sigma_{i=0}^4|H_i|}{2}\right\rceil\right\}$. 
 $\langle\cup_{i=1}^{n}H_i\rangle$ is a complete expansion of $G$. Let us denote this subgraph by $G'$. 
Also, it is not difficult to observe that  $\chi(\mathbb{K}[\mathbb{I}[G]])=\chi(G')$.
By our assumption, $G'$ is $k$-indicated colorable and hence Ann has a winning strategy for  $G'$ using $k$ colors.
Let the color set be $\{1,2,\ldots, k\}$.
%Let the color set be $\{1,2,\ldots, k\}$ where $k\geq\chi(G')$. %$k\geq\chi(\mathbb{K}[\mathbb{I}[G]])$.
%Since $G'$ is $k$-indicated colorable, Ann has a winning strategy for  $G'$. %From Theorem \ref{bipart}, we know that Ann has a winning strategy for  $G$. 
Let Ann present the vertices of $G'$ according to this winning strategy and then presents the remaining vertices of $\mathbb{K}[\mathbb{I}[G]]$ in any order. 
For $1\leq i\leq n$, $V_i$ forms an  independent set and hence for $1\leq j_1< j_2\leq |V_i|$, no vertex in $H_{ij_1}$ is adjacent to any vertex in $H_{ij_2}$. 
Also for $1 \leq j\leq |V_i|$, the neighbors outside the complete expansion of $V_i$ for any two vertices in the complete expansion of $V_i$ is the same
%$N(H_i)=N(H_{ij})$ 
and hence the colors given to the vertices of $H_i$ will be available for the vertices of $H_{ij}$.  
%For $0\le i\leq 4$, $1\leq j\leq |V_i|$, any two vertices in $H_{ij}$ the neighbors in $H_{(i+1)\mathrm{(mod\ 5)} j}\cup H_{(i-1)\mathrm{(mod\ 5)} j}$ are the same, the colors given to the vertices of $H_i$ are available for the uncolored vertices of $H_{ij}$. 
Thus $\mathbb{K}[\mathbb{I}[G]]$ has an indicated coloring using $k$ colors.
\end{proof}
%%%%%%%%%%%%%%%%%%%%%%%%%%%%%%%%%%%%%%%%%%%%%%%%%%%%%%%%%%%%%%%%%%%%%%%%%%%%%%%%%%%%%%%%%%%%%%%%%
%\begin{thm}[\cite{pan}] \label{union}Let $G = G_1 \cup  G_2$. If $G_1$ is $k_1$-indicated colorable for every $k_1 \geq\chi_i(G_1)$ and $G_2$ is
%$k_2$-indicated colorable for every $k_2\geq\chi_i(G_2)$, then $\chi_i(G) = \max\{\chi_i(G_1),\chi_i(G_2)\}$ and $G$ is $k$-indicated
%colorable for all $k\geq\chi_i(G)$.
%\end{thm}
By Theorem \ref{bipart} and Theorem \ref{comexpindexpg}, we see that $\mathbb{K}[\mathbb{I}[C_5]]$ are $k$-indicated colorable for every $k\geq\mathbb{K}[\mathbb{I}[C_5]]$. Hence by using Theorem \ref{union}, Theorem \ref{comexpbi}, Theorem \ref{p5k3} and Theorem \ref{comexpindexpg}, we get Corollary \ref{comexpp5k3}.
%%%%%%%%%%%%%%%%%%%%%%%%%%%%%%%%%%%%%%%%%%%%%%%%%%%%%%%%%%%%%%%%%%%%%%%%%%%%%%%%%%%%%%%%%%%%%%%%%
\begin{cor}\label{comexpp5k3}
  Let $G$ be a $\{P_5,K_3\}$-free graph and $H\cong\mathbb{K}[G](m,m,\ldots,m)$ for some $m\geq1$. Then $\chi_i(H)=\chi(H)$.
\end{cor}
%%%%%%%%%%%%%%%%%%%%%%%%%%%%%%%%%%%%%%%%%%%%%%%%%%%%%%%%%%%%%%%%%%%%%%%%%%%%%%%%%%%%%%%%%%%%%%%%%
By using Theorem \ref{HwithKl} and Corollary \ref{comexpp5k3}, we have Corollary \ref{comexpp5k3product}.
%%%%%%%%%%%%%%%%%%%%%%%%%%%%%%%%%%%%%%%%%%%%%%%%%%%%%%%%%%%%%%%%%%%%%%%%%%%%%%%%%%%%%%%%%%%%%%%%%
\begin{cor}\label{comexpp5k3product}
  Let $G$ be a $\{P_5,K_3\}$-free graph and $H\in\mathcal{H}$ with $\chi_i(H)=\chi(H)$, then $\chi_i(G[H])=\chi(G[H])$.
\end{cor}
%%%%%%%%%%%%%%%%%%%%%%%%%%%%%%%%%%%%%%%%%%%%%%%%%%%%%%%%%%%%%%%%%%%%%%%%%%%%%%%%%%%%%%%%%%%%%%%%%
%%%%%%%%%%%%%%%%%%%%%%%%%%%%%%%%%%%%%%%%%%%%%%%%%%%%%%%%%%%%%%%%%%%%%%%%%%%%%%%%%%%%%%%%%%%%%%%%%

%%%%%%%%%%%%%%%%%%%%%%%%%%%%%%%%%%%%%%%%%%%%%%%%%%%%%%%%%%%%%%%%%%%%%%%%%%%%%%%%%%%%%%%%%%%%%%%%%
Next, let us consider the indicated coloring of the complete expansion of complete multipartite graphs. % of the complete expansion of complete multipartite graph.
%%%%%%%%%%%%%%%%%%%%%%%%%%%%%%%%%%%%%%%%%%%%%%%%%%%%%%%%%%%%%%%%%%%%%%%%%%%%%%%%%%%%%%%%%%%%%%%%%
\begin{thm}\label{comexpcommul}
 Let $G$ be a complete multipartite graph. Then $\mathbb{K}[G]$ is $k$-indicated colorable for all $k\geq\chi(\mathbb{K}[G])$.
\end{thm}
\begin{proof}
For a complete multipartite graph $G$, $\mathbb{K}[G]\cong\mathbb{K}[\mathbb{I}[K_s]]$, for some $s\geq 1$. Also $\mathbb{K}[K_s]$ is isomorphic to a complete graph and hence by using Theorem \ref{comexpindexpg}, we see that  $\mathbb{K}[G]$ is $k$-indicated colorable for all $k\geq\chi(\mathbb{K}[G])$.
%$p\geq2$, let $G$ be the complete $p$-partite graph with partition $V_i$, $1\leq i\leq p$. For $1\leq i\leq p$, let $K_{m_i}$ be the maximal complete subgraph among all the complete subgraphs which have replaced the vertices of $V_i$ in $\mathbb{K}[G]$. 
%Clearly $\omega(\mathbb{K}[G])=\Sigma^{p}_{i=1} m_i$ and $\chi(\mathbb{K}[G])=\omega(\mathbb{K}[G])$. Let the color set be $\{1,2,\ldots,k\geq\chi(\mathbb{K}[G])\}$.
%Let Ann start by presenting the vertices of the maximum clique $\bigcup\limits_{i=1}^{p}K_{m_i}$ in $\mathbb{K}[G]$. Since $k\geq\omega(\mathbb{K}[G])$, there is an available color for each of the presented vertex. 
%Now, let Ann present the remaining vertices of $\mathbb{K}[G]$ in any order. 
%For $1\leq i\leq p$, since the size of the complete subgraphs replaced for the vertices of $V_i$ is at most $m_i$,  the colors given to the vertices of $K_{m_i}$ are available for the remaining vertices in the complete expansion of the vertices of $V_i$. 
%Thus Ann wins the game with $k$ colors.
\end{proof}
%%%%%%%%%%%%%%%%%%%%%%%%%%%%%%%%%%%%%%%%%%%%%%%%%%%%%%%%%%%%%%%%%%%%%%%%%%%%%%%%%%%%%%%%%%%%%%%%%
By using Theorem \ref{HwithKl} and Corollary \ref{comexpcommul}, we have Corollary \ref{comexpcommulproduct}.
%%%%%%%%%%%%%%%%%%%%%%%%%%%%%%%%%%%%%%%%%%%%%%%%%%%%%%%%%%%%%%%%%%%%%%%%%%%%%%%%%%%%%%%%%%%%%%%%%
\begin{cor}\label{comexpcommulproduct}
  Let $G$ be a complete multipartite graph and $H\in\mathcal{H}$ with $\chi_i(H)=\chi(H)$. Then $G[H]$ is $k$-indicated colorable for all $k\geq\chi(G[H])$.
\end{cor}
%%%%%%%%%%%%%%%%%%%%%%%%%%%%%%%%%%%%%%%%%%%%%%%%%%%%%%%%%%%%%%%%%%%%%%%%%%%%%%%%%%%%%%%%%%%%%%%%%
Corollary \ref{comexpp5paw} is an immediate consequence of Corollary \ref{comexpp5k3}, Theorem \ref{paw} and Theorem \ref{comexpcommul}.
%%%%%%%%%%%%%%%%%%%%%%%%%%%%%%%%%%%%%%%%%%%%%%%%%%%%%%%%%%%%%%%%%%%%%%%%%%%%%%%%%%%%%%%%%%%%%%%%%
\begin{cor}\label{comexpp5paw}
Let $G$ be a $\{P_5,Paw\}$-free graph and let $H\cong \mathbb{K}[G](m,m,\ldots,m)$. Then $\chi_i(H)=\chi(H)$.
\end{cor}
%%%%%%%%%%%%%%%%%%%%%%%%%%%%%%%%%%%%%%%%%%%%%%%%%%%%%%%%%%%%%%%%%%%%%%%%%%%%%%%%%%%%%%%%%%%%%%%%%
By using Theorem \ref{HwithKl} and Corollary \ref{comexpp5paw}, we have Corollary \ref{comexpp5pawproduct}.
%%%%%%%%%%%%%%%%%%%%%%%%%%%%%%%%%%%%%%%%%%%%%%%%%%%%%%%%%%%%%%%%%%%%%%%%%%%%%%%%%%%%%%%%%%%%%%%%%
\begin{cor}\label{comexpp5pawproduct}
  Let $G$ be a $\{P_5,Paw\}$-free graph and $H\in\mathcal{H}$ with $\chi_i(H)=\chi(H)$, then $\chi_i(G[H])=\chi(G[H])$.
\end{cor}

\section{Conclusion}
On the whole, Section \ref{conseq} tells us that if $G\in\mathcal{G}=\Big\{$Chordal graphs, Cographs, Complement of bipartite graphs, $\{P_5,C_4\}$-free graphs, connected $\{P_6,C_5,\overline{P_5}, K_{1,3}\}$-free graphs which contain an induced $C_6$, Complete multipartite graphs$\Big\}$ and $H\in\mathcal{H}$ with $\chi_i(H)=\chi(H)$, then $G[H]$ is $k$-indicated colorable for every $k\geq \chi(G[H])$.
By using Theorem \ref{fch}, $\mathcal{F\subseteq H}$ and $\chi_i(H)=\chi(H)$, for every $H\in\mathcal{F}$. Hence  if $G\in\mathcal{G}$ and $H\in\mathcal{F}$,
%=\Big\{$Bipartite graphs, Chordal graphs, Cographs, $\{P_5,K_3\}$-free graphs, $\{P_5,Paw\}$-free graphs, Complement of bipartite graphs, $\{P_5,K_4,Kite,Bull\}$-free graphs, connected $\{P_6,C_5,\overline{P_5},K_{1,3}\}$-free graphs which contain an induced $C_6$, $\mathbb{K}[C_5](m_1, m_2 ,\ldots,m_5)$, $\{P_5,C_4\}$-free graphs, connected $\{P_5,\overline{P_2\cup P_3},\overline{P_5},Dart\}$-free graphs which contain an induced $C_5\Big\}$, 
 then $G[H]$ is $k$-indicated colorable for every $k\geq \chi(G[H])$. 
Also we have shown that  if $G$ is a Bipartite graph or $\{P_5,K_3\}$-free graph (or) $\{P_5,Paw\}$-free graph and $H\in\mathcal{F}$, then  $\chi_i(G[H])=\chi(G[H])$. 
%%%%%%%%%%%%%%%%%%%%%%%%%%%%%%%%%%%%%%%%%%%%%%%%%%%%%%%%%%%%%%%%%%%%%%%%%%%%%%%%%%%%%%%%%%%%%%%%%
%%%%%%%%%%%%%%%%%%%%%%%%%%%%%%%%%%%%%%%%%%%%%%%%%%%%%%%%%%%%%%%%%%%%%%%%%%%%%%%%%%%%%%%%%%%%%%%%%
%\begin{tabular}{|cc|c|c|c|c|l}
%\cline{1-6}
%\multicolumn{2}{ |c| }{\multirow{2}{*}{Heading \newline Heading2}} & \multicolumn{4}{ c| }{Primes} \\ 
%\cline{3-6}
%& & 2 & 3 & 5 & 7 \\ 
%\cline{1-6}
%\multicolumn{1}{ |c  }{\multirow{2}{*}{Powers} } &
%\multicolumn{1}{ |c| }{504} & 3 & 2 & 0 & 1 &     \\ \cline{2-6}
%\multicolumn{1}{ |c  }{}                        &
%\multicolumn{1}{ |c| }{540} & 2 & 3 & 1 & 0 &     \\ \cline{1-6}
%\multicolumn{1}{ |c  }{\multirow{2}{*}{Powers} } &
%\multicolumn{1}{ |c| }{gcd} & 2 & 2 & 0 & 0 & min \\ \cline{2-6}
%\multicolumn{1}{ |c  }{}                        &
%\multicolumn{1}{ |c| }{lcm} & 3 & 3 & 1 & 1 & max \\ \cline{1-6}
%\end{tabular}
%%%%%%%%%%%%%%%%%%%%%%%%%%%%%%%%%%%%%%%%%%%%%%%%%%%%%%%%%%%%%%%%%%%%%%%%%%%%%%%%%%%%%%%%%%%%%%%%%
%%%%%%%%%%%%%%%%%%%%%%%%%%%%%%%%%%%%%%%%%%%%%%%%%%%%%%%%%%%%%%%%%%%%%%%%%%%%%%%%%%%%%%%%%%%%%%%%%
%%%%%%%%%%%%%%%%%%%%%%%%%%%%%%%%%%%%%%%%%%%%%%%%%%%%%%%%%%%%%%%%%%%%%%%%%%%%%%%%%%%%%%%%%%%%%%%%%
%%%%%%%%%%%%%%%%%%%%%%%%%%%%%%%%%%%%%%%%%%%%%%%%%%%%%%%%%%%%%%%%%%%%%%%%%%%%%%%%%%%%%%%%%%%%%%%%%
%%%%%%%%%%%%%%%%%%%%%%%%%%%%%%%%%%%%%%%%%%%%%%%%%%%%%%%%%%%%%%%%%%%%%%%%%%%%%%%%%%%%%%%%%%%%%%%%%
\subsection*{Acknowledgment}
\small For the first author, this research was supported by Post Doctoral Fellowship, Indian Institute of Technology, Palakkad. And for the second author, this research was supported by SERB DST, Government of India, File no: EMR/2016/007339. Also, for the third author, this research was supported by the UGC-Basic Scientific Research, Government of India, Student id: gokulnath.res@pondiuni.edu.in. 
%%%%%%%%%%%%%%%%%%%%%%%%%%%%%%%%%%%%%%%%%%%%%%%%%%%%%%%%%%%%%%%%%%%%%%%%%%%%%%%%%%%%%%%%%%%%%%%%%
%%%%%%%%%%%%%%%%%%%%%%%%%%%%%%%%%%%%%%%%%%%%%%%%%%%%%%%%%%%%%%%%%%%%%%%%%%%%%%%%%%%%%%%%%%%%%%%%%

\end{titlepage}
\end{document}